\documentclass[preprint]{elsarticle}

\usepackage{amssymb}
\usepackage{amsmath}
\usepackage{amsfonts}
\usepackage{amsthm}
\usepackage{times}
\usepackage{color}
\usepackage{xcolor}
\usepackage{graphicx}
\usepackage{tikz}
\usepackage{tikz-cd}
\usetikzlibrary{arrows}
\usepackage{ebproof} 
\usepackage{enumitem} 
\usepackage{mathtools} 
\usepackage{hyperref}
\usepackage{comment}
\usepackage{proof}
\usepackage{multicol}
\usepackage{turnstile}

\newtheorem{thm}{Theorem}[section]
\newtheorem{prop}[thm]{Proposition}
\newtheorem{lemma}[thm]{Lemma}
\newtheorem{cor}[thm]{Corollary}

\theoremstyle{definition}

\newtheorem{remark}[thm]{Remark}

\newcommand{\fontsmall}{\fontsize{7pt}{8pt}\selectfont}
\newcommand{\fontnormal}{\fontsize{10pt}{12pt}\selectfont}

\newcommand{\ld}{\backslash}
\newcommand{\rd}{/}
\newcommand{\set}[1]{\{ #1 \}}
\newcommand{\meet}{\mathbin{\wedge}}
\newcommand{\join}{\mathbin{\vee}}

\newcommand{\alg}[1]{\langle #1 \rangle}

\newcommand{\ext}[1]{\mathbf{R}^S( #1 )}
\newcommand{\extu}[1]{\mathrm{R}^S( #1 )}
\newcommand{\andd}{\mathbin{\&}}
\renewcommand{\phi}{\varphi}
\renewcommand{\epsilon}{\varepsilon}
\newcommand{\hm}{\mathbb{H}}
\newcommand{\iso}{\mathbb{I}}
\newcommand{\pr}{\mathbb{P}}
\newcommand{\pu}{\mathbb{P}_U}
\newcommand{\sub}{\mathbb{S}}
\newcommand{\vr}{\mathbb{V}}
\newcommand{\m}{\bf}
\newcommand{\V}{\mathsf{V}}
\newcommand{\K}{\mathsf{K}}

\makeatletter
\providecommand*{\leftmodels}{%
  \mathrel{%
    \mathpalette\@leftmodels\models
  }%
}
\newcommand*{\@leftmodels}[2]{%
  \reflectbox{$\m@th#1#2$}%
}
\makeatother 

\DeclareMathOperator*{\bang}{!}
\DeclareMathOperator*{\whynot}{?}

\newcommand{\fm}{\mathsf{Fm}}
\newcommand{\eq}{\mathsf{Eq}}
\newcommand{\Var}{\mathsf{Var}}
\newcommand{\var}{\mathsf{var}}

\newcommand{\fle}{\m {FL}_e}

\newcommand{\cfle}{\vdash_{\m {FL}_e}}
\newcommand{\crle}{\vdash_{\m {RL}_e}}
\newcommand{\cll}{\vdash_{\m {LL}}}

\newcommand{\cmall}{\vdash_{\m {MALL}}}
\newcommand{\cflep}{\vdash_{\m {FL}_e}^P}
\newcommand{\crlep}{\vdash_{\m {RL}_e}^P}
\newcommand{\cllp}{\vdash_{\m {LL}}^P}
\newcommand{\cmallp}{\vdash_{\m {MALL}}^P}

\newcommand{\cfles}{\vdash_{\m {FL}_e}^{P+7}}
\newcommand{\crles}{\vdash_{\m {RL}_e}^{P+7}}
\newcommand{\clls}{\vdash_{\m {LL}}^{P+7}}
\newcommand{\cmalls}{\vdash_{\m {MALL}}^{P+7}}

\newcommand{\RL}{\mathsf{RL}}
\newcommand{\PRL}{\mathsf{PRL}}
\newcommand{\GA}{\mathsf{A}}
\newcommand{\GAM}{\mathsf{G}}

\newcommand{\prp}[1]{{\sf #1}}

\journal{}

\begin{document}

\begin{frontmatter}

\title{Interpolation in Linear Logic and Related Systems}

\author[1]{Wesley Fussner}
\ead{fussner@cs.cas.cz}

\author[2]{Simon Santschi}
\ead{simon.santschi@unibe.ch}

\affiliation[1]{organization={Institute of Computer Science, Czech Academy of Sciences}, country={Czech Republic}}
\affiliation[2]{organization={Mathematical Institute, University of Bern}, country={Switzerland}}


\begin{abstract}
We prove that there are continuum-many axiomatic extensions of the full Lambek calculus with exchange that have the deductive interpolation property. Further, we extend this result to both classical and intuitionistic linear logic as well as their multiplicative-additive fragments. None of the logics we exhibit have the Craig interpolation property, but we show that they all enjoy a guarded form of Craig interpolation. We also exhibit continuum-many axiomatic extensions of each of these logics without the deductive interpolation property.
\end{abstract}


\end{frontmatter}

\section{Introduction}

This paper is part of a larger effort to understand interpolation in substructural logics. The latter has been studied extensively in special cases, and complete classifications of substructural logics with various interpolation properties are known for some of these. For instance, Maksimova has shown that there are just $7$ consistent superintuitionsitic logics with the deductive interpolation property (for short, \prp{DIP}) and given concrete axiomatizations of them \cite{Mak77}. On the other hand, the results of \cite{DiNLe00} entail a complete classification of axiomatic extensions of \L{}ukasiewicz logic with the \prp{DIP}, of which there are $\aleph_0$. More generally, a line of work on H\'ajek's basic fuzzy logic shows that there are uncountably many axiomatic extensions of the latter without the \prp{DIP} and offers some partial classifications \cite{Mon06,AB21,FM22}. All of the aforementioned logics are extensions of the full Lambek calculus with the exchange rule $\fle$, and this study's original aim was to ascertain whether $\fle$ has uncountably many axiomatic extensions with the \prp{DIP}. We answer this question in the positive not only for $\fle$, but also for classical and intuitionistic linear logic and their multiplicative-additive fragments. We further show that none of the uncountably many logics we identify have the Craig interpolation property (for short, \prp{CIP}), but that all of them have a guarded variant of the \prp{CIP}. Because the existing literature suggests that interpolation is a relatively rare property among non-classical logics generally and substructural logics in particular, these results are rather surprising. They contribute not just to the study of the particular logics considered here, but also to our understanding of interpolation on a broader level.

Substructural logics comprise a diverse family of resource-sensitive logics, and are often presented in terms of a Gentzen-style sequent calculus called the \emph{full Lambek calculus}, depicted in Figure~\ref{fig:FL}. The full Lambek calculus lacks the basic structural rules (viz. weakening, contraction, and exchange), giving it the flexibility to model many forms of reasoning where these rules may not be valid. The most prominent logics modeled in this way arise as axiomatic extensions of the full Lambek calculus plus the \emph{exchange rule}:
\begin{center}
\mbox{
\infer[(e)]{\Gamma,\alpha,\beta,\Delta\Rightarrow \varphi}{
\Gamma,\beta,\alpha,\Delta\Rightarrow\varphi}
}
\end{center}
The rules ($\cdot l$) and ($\cdot r$) codify that comma separator of the calculus is internalized as the connective $\cdot$, typically called \emph{strong conjunction} or \emph{fusion}. The rule (e) captures the commutativity of this connective, and the full Lambek calculus with exchange $\fle$ encompasses many prominent logical systems of independent origin. These include the most thoroughly studied relevance logics \cite{AB75}, H\'ajek's basic fuzzy logic and its generalizations \cite{Haj98,Fus22,FZ2021,FU2019}, constructive logic with strong negation \cite{SV08,BC2010}, and both the classical and intuitionistic propositional logics. In recent years, axiomatic extensions of $\fle$ have been studied extensively and effectively using the tools of algebraic logic; see \cite{GJKO07} for an overview.

\begin{figure}[t]

\fbox{\begin{minipage}{\textwidth}
\begin{center}
\fontsmall
\medskip
$ 
\begin{array}{ccc} 
\infer[(cut)]{\Gamma,\Sigma,\Delta\Rightarrow \Pi}{
\Sigma \Rightarrow \alpha & 
 \Gamma,\alpha, \Delta\Rightarrow \Pi}
 &
\infer[(id)]{\alpha\Rightarrow \alpha}{} &
\infer[(1 r)]{\Rightarrow 1}{} \\[2em]
\infer[(\cdot l)]{\Gamma,\alpha\cdot\beta , \Delta\Rightarrow \Pi
   }{\Gamma,\alpha,\beta,\Delta\Rightarrow \Pi} &
\infer[(\cdot r)]{\Gamma,\Delta\Rightarrow \alpha\cdot\beta}{
 \Gamma \Rightarrow \alpha & 
 \Delta\Rightarrow \beta} &
\infer[(1 l)]{\Gamma,1,\Delta\Rightarrow \Pi}{
\Gamma,\Delta\Rightarrow \Pi}  \\[2em]
\infer[(\ld l)]{\Gamma,\Sigma,\alpha\ld\beta,\Delta\Rightarrow
  \Pi}{\Sigma \Rightarrow \alpha & \Gamma,\beta,\Delta\Rightarrow \Pi} &
\infer[(\ld r)]{\Gamma\Rightarrow \alpha\ld\beta}{\alpha,
     \Gamma\Rightarrow \beta} &
\infer[(0 l)]{\Gamma\Rightarrow 0}{\Gamma\Rightarrow \ }  \\[2em]
\infer[(\rd l)]{\Gamma,\beta\rd\alpha,\Sigma,\Delta\Rightarrow \Pi}{
   \Sigma \Rightarrow \alpha & 
 \Gamma,\beta,\Delta\Rightarrow \Pi} &
\infer[(\rd r)]{\Gamma\Rightarrow \beta/\alpha}{
\Gamma,\alpha\Rightarrow \beta} &
\infer[(0 r)]{0 \Rightarrow \ }{}  \\[2em]
\infer[(\vee l)]{\Gamma,\alpha\vee\beta,\Delta\Rightarrow \Pi}{
\Gamma,
   \alpha,\Delta\Rightarrow \Pi  &  \Gamma,\beta,\Delta\Rightarrow \Pi} &
\infer[(\vee r_1)]{\Gamma\Rightarrow \alpha\vee\beta}{
\Gamma\Rightarrow \alpha} &
\infer[(\vee r_2)]{\Gamma\Rightarrow \alpha\vee\beta}{
\Gamma\Rightarrow \beta}  \\[2em]
\infer[(\wedge l_1)]{ \Gamma,\alpha\wedge\beta,\Delta\Rightarrow \Pi}{
 \Gamma,\alpha,\Delta\Rightarrow \Pi} &
\infer[(\wedge l_2)]{\Gamma,\alpha\wedge\beta,\Delta\Rightarrow \Pi}{
 \Gamma,\beta,\Delta\Rightarrow \Pi} &
\infer[(\wedge r)]{\Gamma\Rightarrow \alpha\wedge\beta}{
\Gamma\Rightarrow \alpha
    & \Gamma\Rightarrow \beta} 
\end{array}
$
\medskip
\end{center}
\end{minipage}}
\caption{The full Lambek calculus.}
\label{fig:FL}
\end{figure}

Linear logic was introduced by Girard in \cite{girard87} as a logical foundation for parallel computation. In the intervening years, it has been extensively studied and has attracted applications spanning constraint programming \cite{FRS01}, logic programming \cite{Mil96}, and other areas. Its multiplicative-additive fragment ${\bf MALL}$ fits within the framework of the full Lambek calculus expanded by additional constants (see \cite[p.~109]{GJKO07}). Full linear logic ${\bf LL}$ expands the basic language of $\fle$ not only with additional constants and an involutive negation but also by the \emph{exponentials} $\bang$ and $\whynot$, which can be viewed as certain modal operators; see~\cite{girard87,Avr88}. Because of its motivations within the proofs-as-programs paradigm, linear logic has typically been presented and studied in proof-theoretic terms. This approach emphasizes proofs themselves as opposed to the notion of provability, and thus does not employ the framework of Tarskian consequence relations. However, Avron in \cite{Avr88} gave a Hilbert-style proof system for linear logic and studied the corresponding consequence relation, in particular introducing an algebraic semantics. Later, Blok and Pigozzi introduced their powerful framework for the algebraizability of logics in \cite{BP89}, and Aglian\`{o} proved in the 1990s that the consequence relation of linear logic is algebraizable in Blok and Pigozzi's sense.\footnote{We note that Aglian\`{o}'s original work on linear logic was never published and remained folklore among algebraic logicians for more than two decades, but has recently been made available in \cite{AglLL}.}  Thanks to the differing motivations between linear logicians and researchers working within the paradigm of consequence relations, subsequent algebraic studies of linear logic have been sparse. Most semantic studies of linear logic have focused on categorical methods or phase spaces, and have not usually employed the Blok-Pigozzi framework (but cf. \cite{GP23}).

In the present paper, we use algebraic techniques to study interpolation for the consequence relations associated to $\fle$, ${\bf LL}$, and several related systems. In particular, our main result (Theorem~\ref{t:main}) gives that each of these systems has continuum-many axiomatic extensions with the deductive interpolation property. Our results rest on algebraization for the aforementioned deductive systems \cite{Galatos2006,AglLL}, as well as the fact that each of these deductive systems has the deductive interpolation property if and only if its associated class of algebraic models has the amalgamation property \cite{CP99}. Accordingly, we construct continuum-many varieties (AKA equational classes) of the respective algebraic models that have the amalgamation property. We obtain these varieties by constructing generating algebras for them from appropriately chosen abelian groups. Our construction of these algebraic models is tailored to exploit the fact that the category of abelian groups is replete with injective objects (see Theorem~\ref{l:embed-inj}), which is the key to getting the amalgamation property for the varieties we identify.

The aforementioned results demonstrate that $\fle$, ${\bf LL}$, and multiplicative-additive linear logic are rich in extensions with the \prp{DIP}. To provide further context, we show in Theorem~\ref{t:noDIP} that each of these logics also has continuum-many axiomatic extensions without the \prp{DIP}. This was known previously in the case of $\fle$ (see \cite{Mak77}), but the result is new for linear logic and its multiplicative-additive fragment. Together with the previously articulated positive results, this demonstrates that $\fle$, ${\bf LL}$, and several related systems provide environments where both the presence and absence of interpolation is abundant.

The methods we use rely crucially on the tools of algebraic logic, as well as the special features of abelian groups. Because of this, we do not expect that our results can be easily reproduced by other techniques, in particular proof-theoretic ones. We thus regard the present study not only as a contribution to the study of interpolation in substructural logics, but also as illustrating the potential of algebraic tools in linear logic. Commensurately, we hope that this study is of broad interest beyond specialists in algebraic logic. In order to make this work as widely accessible as possible---and because Aglian\`{o}'s results on algebraic linear logic have only recently become available---we provide an especially thorough discussion of preliminaries in Section~\ref{sec:deductive systems}. We hope that this provides a guide allowing general logicians to appreciate the key features of the arguments, as well as suggest additional avenues for the application of algebraic techniques.

\section{Linear logic and algebraic logic}\label{sec:deductive systems}
Our investigation is based in algebraic substructural logic, regarding which we presently recall some pertinent facts. Our treatment aims to compactly present the highlights of algebraic substructural logics to non-specialists, while also specifying the classes of algebraic structures we consider. For further information on algebraic substructural logic, we refer to \cite{GJKO07}. We direct the reader to \cite{F2016} for background on abstract algebraic logic generally, and for preliminaries on universal algebra see \cite{BS81}.

\subsection{Deductive systems and the logics in question} As usual, we fix a countably infinite set of propositional variables $\Var$. Given an algebraic language $\mathcal{L}$ (whose members serve as logical connectives), we denote the collection of formulas constructed from $\mathcal{L}$ and variables in $\Var$ by $\fm_\mathcal{L}$. If $\varphi\in\fm_\mathcal{L}$, then $\var(\varphi)$ stands for the collection of variables appearing in the formula $\varphi$; we extend this notation to any collection $\Gamma\subseteq\fm_\mathcal{L}$ by setting $\var(\Gamma) = \bigcup\{\var(\varphi) : \varphi\in\Gamma\}$. A \emph{substitution} is an endomorphism of the absolutely free algebra ${{\m \fm}_\mathcal{L}}$ over the language $\mathcal{L}$. Further, a \emph{consequence relation} over $\mathcal{L}$ is a relation $\vdash\;\subseteq \mathcal{P}(\fm_\mathcal{L})\times\fm_\mathcal{L}$ from sets of formulas to formulas satisfying, for any $\Gamma\cup \Pi\cup \{ \varphi \}\subseteq\fm_\mathcal{L}$:
\begin{enumerate}
\item If $\varphi\in\Gamma$, then $\Gamma\vdash\varphi$ (\emph{reflexivity});
\item If $\Gamma\vdash \varphi$ and $\Gamma\subseteq\Pi$, then $\Pi\vdash\varphi$ (\emph{monotonicity});
\item If $\Gamma\vdash \varphi$ and $\Pi\vdash\psi$ for every $\psi\in\Gamma$, then $\Pi\vdash\varphi$ (\emph{transitivity});
\item If $\Gamma\vdash \varphi$, then $\sigma[\Gamma]\vdash\sigma(\varphi)$ for every substitution $\sigma$ (\emph{structurality}).
\end{enumerate}
A \emph{deductive system} is a consequence relation $\vdash$ that is \emph{finitary} in the sense that:
\begin{enumerate}
\item[5.] If $\Gamma\vdash\varphi$, then there exists a finite subset $\Gamma'\subseteq\Gamma$ such that $\Gamma'\vdash\varphi$.
\end{enumerate}
If $\vdash$ is a consequence relation over the language $\mathcal{L}$, a formula $\varphi\in\fm_\mathcal{L}$ is called a
 \emph{theorem} of $\vdash$ provided that $\vdash\varphi$. If $\Gamma,\Sigma\subseteq\fm_\mathcal{L}$ and $\vdash$ is a consequence relation over $\mathcal{L}$, then we write $\Gamma\vdash\Sigma$ provided that $\Gamma\vdash\varphi$ for every $\varphi\in\Sigma$. Moreover, we write $\Gamma\dashv \, \vdash \Sigma$ if $\Gamma\vdash\Sigma$ and $\Sigma\vdash\Gamma$.

The term `logic' has several precise meanings in the literature, sometimes referring to deductive systems and at other times to sets of formulas, particular proof-theoretic presentations of these, and so forth. These differing levels of description come with varying degrees of specification: A deductive system may be presented by many different proof-theoretic formalisms, and different deductive systems may have the same theorems. Due to connection between computations and proofs, the literature on linear logic typically conceptualizes `logic' as referring to particular calculi. In this paper, we adopt a more consequence-driven perspective and use the term \emph{logic} to mean `deductive system'.

Suppose that $\vdash$ is a logic over a language $\mathcal{L}$. A logic $\vdash^\ast$ over $\mathcal{L}$ is an \emph{axiomatic extension} of $\vdash$ if there exists a set of formulas $\Sigma\subseteq\fm_\mathcal{L}$, where $\Sigma$ is closed under substitutions, such that 
\[
\Gamma \vdash^\ast \phi \iff \Gamma, \Sigma \vdash \phi.
\]
Intuitively, an axiomatic extension of a logic $\vdash$ simply arises from adjoining new axiom schemes to $\vdash$ and closing the resulting set under consequence. We note that each axiomatic extension of a logic $\vdash$ over $\mathcal{L}$ is a subset of $\mathcal{P}(\fm_\mathcal{L})\times\fm_\mathcal{L}$, and the collection of axiomatic extensions of $\vdash$ forms a lattice under inclusion.

We are concerned with certain deductive systems over supersets of the basic language $\mathcal{RL}=\{\meet,\join,\cdot,\to,1\}$, where $\meet,\join,\cdot,\to$ are binary and $1$ is nullary. In particular, we use $0,\bot,\top$ to denote nullary function symbols/connectives and $\bang,\whynot$ to denote unary ones; we will consider several deductive systems over languages $\mathcal{RL}\cup S$, where $S\subseteq \{0,\bot,\top,\bang,\whynot\}$.

Figure~\ref{fig:calculus} depicts Avron's Hilbert-style calculus for classical linear logic \cite{Avr88}. Here $\neg \alpha$ abbreviates $\alpha \to 0$ and as usual the connectives $\meet,\join,\cdot,\bang,\neg$ bind stronger than $\to$; the connective $\whynot$ is definable in this system as an abbreviation for $\neg\bang\neg$, so need not be included in the language. Considered over the language $\mathcal{RL}\cup\{0\}$, the calculus defined by (A1)--(A13) together with (mp) and (adj) axiomatizes the deductive system $\cfle$ of the full Lambek calculus with exchange; its $0$-free fragment $\crle$ is given by working over the language $\mathcal{RL}$ (cf. \cite{GJKO07}).
\footnote{Spelling this out, suppose that $\Gamma \cup \{ \varphi \}$ is a set of formulas in the language. Then the sequent $\Rightarrow\varphi$ is provable from the set of sequents $\{\Rightarrow \psi : \psi \in \Gamma \}$ in the full Lambek calculus (see Figure~\ref{fig:FL}) if and only if $\varphi$ is provable from $\Gamma$ in the Hilbert calculus defined by (A1)--(A13) together with (mp) and (adj).} Adding (A$\bot$) and (A$\top$) and working over $\mathcal{RL}\cup\{0,\bot,\top\}$, we obtain a calculus for the full Lambek calculus with bounds. Further adding (A0)--(Con) yields the deductive system $\cmall$ of classical multiplicative-additive linear logic. The deducibility relation $\cll$ for full classical linear logic is specified by further adding the axiom schemes ($\bang$w)--($\bang$4) and the rule (nec).

Different notational conventions appear in the literature, with some especially significant differences between the substructural logic and linear logic communities. Here we primarily use notation that is common among substructural logicians, but we keep Girard's notation from \cite{girard87} for $\bang$ and $\whynot$. We differ from Avron \cite{Avr88} on this point; observing that $\bang$ is tantamount to the necessity operator of an S4-style modal logic, he uses $\Box$ in place of $\bang$. However, the naming convention we have adopted for the axiom schemes in Figure~\ref{fig:calculus} reflects the naming conventions from modal logic. Figure~\ref{fig:notation} indicates how to toggle between the linear logicians' notation from \cite{girard87} and our notation.
\begin{figure}[t]
\fbox{
\begin{minipage}[t]{\textwidth}
\fontsmall
\begin{center}
{\bf \fontnormal Axiom schemes}\\
\setlength{\columnsep}{0 cm}
\fontsmall
\begin{multicols}{2}
\setlength{\itemsep}{0 cm}
    \begin{itemize}[align=left]
	\item[(A1)] $\alpha \to \alpha$
	\item[(A2)] $\alpha \meet \beta \to \alpha$
	\item[(A3)] $\alpha \meet \beta \to \beta$
	\item[(A4)] $\alpha \to \alpha \join \beta$
	\item[(A5)] $\beta \to \alpha \join \beta$
	\item[(A6)]  $(\alpha \to \beta ) \to ((\beta \to \gamma) \to (\alpha \to  \gamma)$
	\item[(A7)] $(\alpha \to (\beta \to \gamma)) \to (\beta \to (\alpha \to \gamma))$
	\item[(A8)] $(\alpha \to \beta) \meet (\alpha \to \gamma) \to ( \alpha \to \beta \meet \gamma)$
	\item[(A9)] $(\alpha \to \gamma) \meet (\beta \to \gamma) \to (\alpha \join \beta \to \gamma)$
	\item[(A10)] $\alpha \to ( \beta \to \alpha \cdot \beta)$
	\item[(A11)] $(\alpha \to (\beta \to \gamma)) \to (\alpha \cdot \beta \to \gamma)$
	\item[(A12)] $1$
	\item[(A13)] $1 \to (\alpha \to \alpha)$
	\item[(A$\bot$)] $\alpha \to \top$
	\item[(A$\top$)] $\bot \to \alpha$
	\item[(A0)] $\neg 0$
	\item[(NC)] $\alpha \to (\neg\alpha \to 0)$
	\item[(DN)] $\neg\neg \alpha \to \alpha$
	\item[(Con)] $(\alpha \to \neg\beta) \to (\beta \to \neg \alpha)$
	\item[($\bang$w)] $\beta \to (\bang \alpha \to \beta)$
	\item[($\bang$i)] $(\bang \alpha \to (\bang \alpha \to \beta)) \to (\bang \alpha \to \beta)$
	\item[($\bang$K)] $\bang(\alpha \to \beta) \to (\bang \alpha \to \bang \beta)$
	\item[($\bang$T)] $\bang \alpha \to \alpha$
	\item[($\bang$4)] $\bang \alpha \to \bang\bang \alpha$
\end{itemize}
\end{multicols}
\end{center}

\begin{center}
{\bf \fontnormal Rules of inference}\\
\;\\

\fontsmall
 \begin{prooftree} \hypo{\alpha}\hypo{\alpha \to \beta}\infer2[(mp)]{\beta} \end{prooftree} \qquad \begin{prooftree} \hypo{\alpha}\hypo{\beta}\infer2[(adj)]{\alpha \meet \beta} \end{prooftree} \qquad \begin{prooftree} \hypo{\alpha}\infer1[(nec)]{\bang\alpha} \end{prooftree} 
\end{center}
\end{minipage}}
\caption{A Hilbert-style calculus for linear logic and related systems}
\label{fig:calculus}
\end{figure}

\begin{figure}[t]
\begin{center}
\fbox{
\begin{minipage}[t]{6cm}
\begin{tabular}{ccc}
\textbf{Girard's Notation}  &  & \textbf{Our Notation} \\
\hline
$\&$ & & $\meet$ \\
$\oplus$ & &$\join$ \\
$\multimap$ & & $\to$ \\
$( \, )^\bot$& & $\neg$ \\
$\otimes$ & & $\cdot$ \\
$1,\bot$ & &$1,0$ \\
$0, \top$ & &$\bot,\top$
\end{tabular}
\end{minipage}}
\end{center}
\caption{Correspondence between different notational conventions}
\label{fig:notation}
\end{figure}

In general, substructural logics with exchange only have local deduction theorems; see \cite{Galatos2006}. This is true, for example, of both $\cfle$ and $\cmall$. In contrast, $\cll$ has an explicit deduction theorem.

\begin{thm}[Avron's Deduction Theorem (\cite{Avr88})]\label{t:Avron}
Let $\Gamma\cup\{\varphi,\psi\}$ be a set of formulas in the language of linear logic. Then:
\begin{center}
$\Gamma,\varphi\cll \psi \iff \Gamma\cll \bang\varphi\to \psi$.
\end{center}
\end{thm}

\begin{remark}
Clearly, Avron's Deduction Theorem also holds for arbitrary axiomatic extensions of classical linear logic. This may also be proven semantically by applying the algebraization results summarized in Section~\ref{sec:algebraization}.
\end{remark}


\subsection{Algebraization}\label{sec:algebraization}
We will consider algebraic counterparts of the deductive systems just introduced. We first discuss the pertinent classes of algebraic structures, and subsequently discuss their connection to the deductive systems $\cll$, $\cmall$, $\cfle$, and $\crle$.

Given a class $\sf K$ of similar algebras, we denote by $\hm(\mathsf{K})$, $\sub(\mathsf{K})$, $\pr(\mathsf{K})$, $\pu(\mathsf{K})$, and $\iso(\mathsf{K})$ its closure under taking homomorphic images, subalgebras, products, ultraproducts, and isomorphic images, respectively. We recall that a \emph{variety} is a class of similar algebras modeling a given collection of equations (formally defined below). A class $\sf K$ of similar algebras is a variety if and only if $\sf K$ is closed under $\hm$, $\sub$, and $\pr$. In fact, if $\sf K$ is any class of similar algebras, the least variety containing $\sf K$ coincides with $\hm\sub\pr(\mathsf{K})$. We denote the variety $\hm\sub\pr(\mathsf{K})$ generated by $\sf K$ by $\vr(\mathsf{K})$. 
Recall also that a  class  $\sf K$ of similar algebras is called a \emph{quasivariety} if it axiomatized by a set of quasi-equations or, equivalently, if $\iso\sub\pr\pu(\mathsf{K}) = \mathsf{K}$.

A \emph{commutative residuated lattice} is an algebra of the form $\alg{A,\meet,\join,\cdot,\to 1}$, where $\alg{A,\meet,\join}$ is a lattice, $\alg{A,\cdot,1}$ is a commutative monoid, and for all $a,b,c\in A$,
\[
a\cdot b\leq c \iff a\leq b\to c,
\]
where $\leq$ is the partial order corresponding to the lattice operations (that is, $a\leq b$ if and only if $a\meet b=a$). Note that this last condition can be replaced by a finite set of identities, so commutative residuated lattices form a finitely axiomatizable variety. A \emph{pointed commutative residuated lattice} is a commutative residuated lattice with an extra constant $0$ added to its signature.

When the underlying poset of a (pointed) commutative residuated lattice is bounded, we often include constant symbols $\bot$ and $\top$ in the signature, respectively denoting the least and greatest elements of the poset. We refer to the resulting structure in this enriched signature as a \emph{bounded (pointed) commutative residuated lattice}. 

An \emph{involution} on a commutative residuated lattice $\alg{A,\meet,\join,\cdot,\to 1}$ is a unary operation $\neg$ on $A$ such that, for all $a,b\in A$, $\neg\neg a = a$ and $a\to\neg b = b\to\neg a$. If $\neg$ is an involution on a commutative residuated lattice $\alg{A,\meet,\join,\cdot,\to, 1}$, then we often abbreviate $\neg 1$ by $0$. In this event we have $a=(a\to 0)\to 0$ for all $a\in A$. We call an arbitrary constant $0$ satisfying the latter identity a \emph{negation constant}. If $0$ is a negation constant in some commutative residuated lattice $\alg{A,\meet,\join,\cdot,\to,0}$, then one may define an involution $\neg$ by setting $\neg a = a\to 0$. Consequently, the expansion of a commutative residuated lattice $\alg{A,\meet,\join,\cdot,\to,0}$ by an involution is term-equivalent to its expansion by the negation constant $0=\neg 1$. Either of these term-equivalent expansions is called an \emph{involutive commutative residuated lattice}. Note that this terminology introduces ambiguity, but presents no serious problems to this study. We call a bounded involutive commutative residuated lattice an \emph{A-algebra}.\footnote{The `A' in A-algebra stands for both Aglian\`{o} and Avron.}

An algebra $\alg{A,\meet,\join,\cdot,\to,1,0,\bot,\top,\bang}$ is a \emph{girale} if its $\alg{A,\meet,\join,\cdot,\to,1,0,\bot,\top}$ is an A-algebra and $\bang$  is a unary operation on $A$ such that for all $a,b\in A$:
\begin{enumerate}[label = \textup{(G\arabic*)}]
\item $\bang 1 = 1$.
\item $\bang a \leq a\meet 1$.
\item $\bang a\cdot\bang b = \bang (a\meet b)$.
\item $\bang\bang a = \bang a$.
\end{enumerate}

We denote by $\RL$, $\PRL$, $\GA$, and $\GAM$, the varieties of commutative residuated lattices, pointed commutative residuated lattices, A-algebras, and girales, respectively. Each of these comprises an arithmetical variety with the congruence extension property.\footnote{A variety is called \emph{arithmetical} if it is both congruence distributive and congruence permutable.}

The classes of algebraic structures just introduced are pertinent to our study because they give algebraic models for the deductive systems of Section~\ref{sec:deductive systems}. To say what this means precisely, we must make formal several notions. Given an algebraic language $\mathcal{L}$, an \emph{equation} over $\mathcal{L}$ is a pair $\alg{s,t}\in \fm_\mathcal{L}^2$. We will write an equation $\alg{s,t}$ as $s\approx t$, and denote by $\eq_\mathcal{L}$ the collection of equations constructed from $\mathcal{L}$ along with the variables from $\Var$. If $\m A$ is an algebra in the language $\mathcal{L}$, an \emph{assignment} into ${\m A}$ is a homomorphism $h$ from the absolutely free algebra over $\mathcal{L}$ to $\m A$. For a class ${\sf K}$ of algebras in the language $\mathcal{L}$, we define a relation $\models_\K\;\subseteq\mathcal{P}(\eq_\mathcal{L})\times\eq_\mathcal{L}$ from sets of equations to equations by
\begin{align*}
E\models_\K (u\approx w) \iff &\text{ For each }{\m A}\in\K\text{ and each assignment }h\text{ into }{\m A},\\
& h(u)=h(w) \text{ whenever }h(s)=h(t)\text{ for all }(s\approx t)\in E.
\end{align*}
The relation $\models_\K$ is called the \emph{equational consequence relation} of $\K$. We adopt all the expected notation, writing $E\leftmodels \models_{\sf K} S$ for $E\models_\K S$ and $S\models_\K E$, and so on.

A deductive system $\vdash$ over $\mathcal{L}$ is called \emph{algebraizable} (see \cite{BP89}) if there exist a finite set of equations $\tau(x)$ in one variable $x$, a finite set of formulas $\Delta(x,y)$ in two variables $x,y$, and a quasivariety $\sf K$ such that
\begin{align*}
\Gamma \vdash \phi &\iff \tau[\Gamma] \models_{\sf K} \tau(\phi) \\
\Theta \models_{\sf K} \epsilon \approx \delta &\iff \Delta[\Theta] \vdash \Delta(\epsilon,\delta) \\
\phi  & \  \dashv \, \vdash \Delta[\tau(\phi)] \\
\epsilon \approx \delta & \  \leftmodels \models_{\sf K} \tau[\Delta(\epsilon,\delta)]
\end{align*}
for every set of formulas $\Gamma \cup \{\varphi\}\subseteq\fm_\mathcal{L}$ and set of equations $\Theta \cup \set{\epsilon \approx \delta}\subseteq\eq_\mathcal{L}$. In this situation, the equations $\tau(x)$ are called \emph{defining equations}, the formulas $\Delta(x,y)$ are called \emph{equivalence formulas}, and $\K$ is called the \emph{equivalent algebraic semantics} of $\vdash$. We say that a deductive system $\vdash$ is \emph{strongly algebraizable} when it is algebraizable and its equivalent algebraic semantics is a variety.

When $\K$ is the equivalent algebraic semantics of a deductive system $\vdash$, the defining equations and equivalence formulas witness mutually inverse translations between $\vdash$ and the equational consequence of $\K$. This creates a powerful link between $\vdash$ and $\K$ that facilitates a back-and-forth transfer of properties between the two. The following well-known result is one such link that we will exploit in the present work.
\begin{prop}[{\cite[Corollary 3.40]{F2016}}]\label{prop:isomorphism}
Let $\vdash$ be an strongly algebraizable deductive system and let $\V$ be its equivalent algebraic semantics. Then the lattice of axiomatic extensions of $\vdash$ is dually isomorphic to the lattice of subvarieties of $\V$.
\end{prop}

Galatos and Ono showed in \cite{Galatos2006} that $\cfle$ and several related deductive systems are strongly algebraizable, and substructural logics have been studied quite extensively using these methods. Algebraic semantics have been deployed much less in linear logic. Avron introduced an algebraic semantics for linear logic in \cite{Avr88}, but did not demonstrate the mutual interpretability required of an equivalent algebraic semantics. In an unpublished manuscript (updated in \cite{AglLL}), Aglian\`{o} showed that $\cmall$ and $\cll$ are algebraizable and identified the varieties of A-algebras and girales in a term-equivalent signature.\footnote{We note that the modal nature of $\bang$ is key in the algebraization of $\cll$. Referring to Figure~\ref{fig:calculus}, the presence of the axiom scheme ($\bang$K) and the strong form of the necessitation rule (nec) are especially crucial; see the discussion of algebraization of the global consequence relation of normal modal logics in \cite[pp.~46--47]{BP89}.} The following theorem summarizes information regarding algebraizability that we will use in the sequel.

\begin{thm}\label{t:algebraizable}
Let $\delta(x) = x\meet 1 \approx 1$ and $\Delta(x,y) = \set{x\to y, y \to x}$.
\begin{enumerate}[label = \textup{(\roman*)}]
\item The system $\crle$ is strongly algebraizable with defining equation $\delta(x)$, equivalence formulas $\Delta(x,y)$, and equivalent algebraic semantics $\RL$.
\item The system $\cfle$ is strongly  algebraizable with defining equation $\delta(x)$, equivalence formulas $\Delta(x,y)$, and equivalent algebraic semantics $\PRL$.
\item  The system $\cmall$ is  strongly algebraizable with defining equation $\delta(x)$, equivalence formulas $\Delta(x,y)$, and equivalent algebraic semantics $\GA$.
\item  The system $\cll$ is strongly  algebraizable with defining equation $\delta(x)$, equivalence formulas $\Delta(x,y)$, and equivalent algebraic semantics $\GAM$.
\end{enumerate}
\end{thm}

\begin{remark}
Since the equivalent algebraic semantics of the deductive systems $\crle$, $\cfle$, $\cmall$, $\cll$ are varieties, Proposition~\ref{prop:isomorphism} gives an anti-isomorphism between their lattices of axiomatic extensions and corresponding subvariety lattices of their equivalent algebraic semantics.
\end{remark}

\subsection{Interpolation, amalgamation, and injectivity} An arbitrary logic $\vdash$ is said to have the \emph{deductive interpolation property} (or \prp{DIP}) if whenever $\Gamma\vdash\varphi$, there exists a set of formulas $\Gamma'$ such that $\var(\Gamma')\subseteq\var(\Gamma)\cap\var(\varphi)$ and $\Gamma\vdash\Gamma'$, $\Gamma'\vdash\varphi$. The collection of formulas $\Gamma'$ is said to be a \emph{deductive interpolant}. 

\begin{remark}
What we call the deductive interpolation property is sometimes called the ``weak deductive interpolation property''. There is also a stronger notion: A deductive system is then said to have the \emph{strong deductive interpolation property} if  whenever $\Gamma,\Delta\vdash\varphi$, there exists a set of formulas $\Gamma'$ such that $\var(\Gamma')\subseteq\var(\Gamma)\cap\var(\Delta \cup \set{\varphi})$, $\Gamma\vdash\Gamma'$, and $\Gamma',\Delta\vdash\varphi$. However, for the deductive systems we consider the two notions are equivalent, since the systems are substructural logics with exchange (cf. \cite{GJKO07}). 
\end{remark}

We say that a consequence relation $\vdash$ is \emph{conjunctive} if there exists a formula in two variables $\kappa(x,y)$ such that for every set of  formulas $\Gamma \cup \set{\phi,\psi,\gamma}$, 
\[
\Gamma, \phi,\psi \vdash \gamma \iff \Gamma, \kappa(\phi,\psi) \vdash \gamma.
\] 
\begin{remark}\label{r:DIP-one-formula}
If a consequence relation $\vdash$ is finitary and conjunctive, it has the \prp{DIP} if and only if for all formulas $\phi,\psi$, whenever $\phi \vdash \psi$, there exists a formulas $\delta$ such that  $\var(\delta)\subseteq\var(\phi)\cap\var(\psi)$ and $\phi\vdash \delta$, $\delta \vdash \phi$.  This holds in particular for any axiomatic extension of the logics $\cll$, $\cmall$, $\cfle$, and $\crle$  with  $\kappa(x,y) = x \meet y$.
\end{remark}

When $\vdash$ is a logic over a language that includes an implication connective $\to$, we say that $\vdash$ has the \emph{Craig interpolation property} (or \prp{CIP}) if whenever $\vdash \varphi\to \psi$, there exists a formula $\delta$ such that $\vdash\varphi\to\delta$, $\vdash\delta\to\psi$, and the variables appearing in $\delta$ are among those appearing in both of $\varphi$ and $\psi$. In this event, we call $\delta$ a \emph{Craig interpolant}. It is shown in \cite{Galatos2006} that if $\vdash$ is a substructural logic with exchange and $\vdash$ has the \prp{CIP}, then it has the \prp{DIP} as well. The converse is not true in general. Intuitively, interpolants provide an explanation why a particular inference holds; see \cite{HH99}.

\begin{remark}
For the logics  $\cll$, $\cmall$, $\cfle$, and $\crle$ the Craig interpolation property can be proved syntactically using a suitable cut-free Gentzen-style sequent calculus; see \cite{GJKO07,Roorda1994}.  In particular, this implies that these logics the \prp{DIP}.
\end{remark}

The algebraic counterpart of the deductive interpolation property is the amalgamation property.

A \emph{span} in $\sf K$ is a quintuple $\alg{\mathbf{A}, \mathbf{B},\mathbf{C},\phi_1,\phi_2}$ of algebras $\mathbf{A},\mathbf{B}, \mathbf{C} \in \mathsf{K}$ and embeddings $\phi_1\colon \mathbf{A} \to \mathbf{B}$ and $\phi_2 \colon \mathbf{A} \to \mathbf{C}$. 
An \emph{amalgam} of a span $\alg{\mathbf{A}, \mathbf{B},\mathbf{C},\phi_1,\phi_2}$ in $\sf K$ is a triple $\alg{\psi_1,\psi_2,\mathbf{D}}$ where $\mathbf{D} \in \mathsf{K}$ and $\psi_1\colon \mathbf{B} \to \mathbf{D}$, $\psi_2 \colon \mathbf{C} \to \mathbf{D}$ are embedding such that $\psi_1\circ \phi_1 = \psi_2 \circ \phi_2$, i.e., the diagram in Figure~\ref{fig:AP}(i) commutes. The amalgam $\alg{\psi_1,\psi_2,\mathbf{D}}$ is called a \emph{strong amalgam} if, moreover, we have $(\psi_1\circ \phi_1)[A] = \psi_1[B] \cap \psi_2[C]$.
The class $\sf K$ is said to have the \emph{(strong) amalgamation property} if every span in $\sf K$ has a (strong) amalgam.

\begin{figure}
\centering
\begin{tabular}{ccc}
\begin{tikzcd}
	& {\bf B} \\
	{\bf A} && {\bf D} \\
	& {\bf C}
	\arrow["{\phi_1}", hook, from=2-1, to=1-2]
	\arrow["{\psi_1}", dashed, hook, from=1-2, to=2-3]
	\arrow["{\phi_2}"', hook, from=2-1, to=3-2]
	\arrow["{\psi_2}"', dashed, hook, from=3-2, to=2-3]
\end{tikzcd}
& &
\begin{tikzcd}
	& {\bf B} \\
	{\bf A} && {\bf D} \\
	& {\bf C}
	\arrow["{\phi_1}", hook, from=2-1, to=1-2]
	\arrow["{\psi_1}", dashed, hook, from=1-2, to=2-3]
	\arrow["{\phi_2}"', hook, from=2-1, to=3-2]
	\arrow["{\psi_2}"', dashed,  from=3-2, to=2-3]
\end{tikzcd}
\\
(i) & & (ii)
\end{tabular}
\caption{Commutative diagrams for the amalgamation properties}
\label{fig:AP}
\end{figure}

A \emph{one-sided amalgam} of a span $\alg{\mathbf{A}, \mathbf{B},\mathbf{C},\phi_1,\phi_2}$ in $\sf K$ is a triple $\alg{\psi_1,\psi_2,\mathbf{D}}$ where $\mathbf{D} \in \mathsf{K}$ and $\psi_1\colon \mathbf{B} \to \mathbf{D}$, is an embedding and  $\psi_2 \colon \mathbf{C} \to \mathbf{D}$ is a homomorphism such that $\psi_1\circ \phi_1 = \psi_2 \circ \phi_2$, i.e., the diagram in Figure~\ref{fig:AP}(ii) commutes.
The class $\sf K$ is said to have the \emph{one-sided amalgamation property} if every span in $\sf K$ has a one-sided amalgam.

In order to establish that a variety $\V$ has the amalgamation property, it is often convenient to work with a tractable generating class for $\V$. Recall that an algebra ${\m A}$ is said to be \emph{finitely subdirectly irreducible} if the least congruence $\Delta$ of ${\m A}$ is meet-irreducible, i.e., whenever $\Theta$ and $\Psi$ are congruences of ${\m A}$ and $\Delta = \Theta\cap\Psi$ we have $\Delta=\Theta$ or $\Delta=\Psi$. The following may be found in \cite[Theorem 3.4]{FM22}. 

\begin{thm}\label{t:lift amalg}
Suppose that $\V$ is a variety with the congruence extension property, and assume that the class $\V_{\mathrm{FSI}}$ of finitely subdirectly irreducible members of $\V$ is closed under subalgebras.  Then  $\V$ has the amalgamation property if and only if $\V_{\mathrm{FSI}}$ has the one-sided amalgamation property.
\end{thm}

Let $\mathsf{K}$ be a class of similar algebras. An algebra $\mathbf{Q} \in \mathsf{K}$ is called \emph{injective} over $\mathsf{K}$ if for all algebras $\mathbf{A}, \mathbf{B} \in \mathsf{K}$, embedding $\alpha \colon \mathbf{B} \to \mathbf{A}$, and homomorphism $\beta \colon \mathbf{B} \to \mathbf{Q}$ there exists a homomorphism $\phi \colon \mathbf{A} \to \mathbf{Q}$ such that $\phi\circ \alpha = \beta$, i.e., the following diagram commutes:
\[\begin{tikzcd}
	{\bf B} & {\bf A} \\
	{\bf Q}
	\arrow["\beta"', from=1-1, to=2-1]
	\arrow["\phi", dashed, from=1-2, to=2-1]
	\arrow["\alpha", hook, from=1-1, to=1-2]
\end{tikzcd}\]
The class $\mathsf{K}$ is said to have \emph{enough injectives} if every algebra in $\mathsf{K}$ embeds into an algebra in $\mathsf{K}$ that is injective over $\mathsf{K}$.

\begin{lemma}[cf. \cite{KMPT83}]\label{l:EI-AP}
Let $\mathsf{K}$ be a class of similar algebras that is closed under taking finite products. If $\mathsf{K}$ has enough injectives, then it has the amalgamation property. 
\end{lemma}

Let $\sf K$ be a class of similar algebras and $\mathbf{A},\mathbf{B} \in \mathsf{K}$. A homomorphism $\phi\colon \mathbf{A} \to \mathbf{B}$ is called an \emph{epimorphism in} $\sf K$ if for all $\mathbf{C} \in \mathsf{K}$ and homomorphisms $\psi_1,\psi_2 \colon \mathbf{B} \to \mathbf{C}$, if $\psi_1\circ \phi = \psi_2 \circ \phi$, then $\psi_1 = \psi_2$. The class $\mathsf{K}$ is said to have \emph{surjective epimorphisms} if every epimorphism in $\sf K$ is surjective.

\begin{lemma}[cf. \cite{KMPT83}]\label{l:ES-SAP}
A quasivariety $\mathsf{K}$ has surjective epimorphisms and the amalgamation property if and only if  it has the strong amalgamation property.
\end{lemma}

The following summarize some well-known bridge theorems. To properly state these, we recall that a deductive system $\vdash$ over $\mathcal{L}$ has a \emph{local deduction theorem} if there exists a family $\Lambda$ of sets of formulas in two variables such that, for all $\Gamma\cup\{\varphi,\psi\}\subseteq\fm_\mathcal{L}$, we have $\Gamma,\varphi\vdash\psi$ if and only if there exists $\lambda\in\Lambda$ such that for all $l\in\lambda$, $\Gamma\vdash l(\varphi,\psi)$. Local deduction theorems generalize explicit deduction theorems, such as Avron's deduction theorem for linear logic.

\begin{thm}[\cite{Blok1991}]\label{t:Blok1991}
Let $\vdash$ be a strongly algebraizable deductive system with equivalent algebraic semantics $\sf V$. Then $\vdash$ has a local deduction theorem if and only if $\sf V$ has the congruence extension property.
\end{thm}

\begin{thm}[\cite{CP99}]\label{t:DIP-AP}
Let $\vdash$ be a strongly algebraizable deductive system with a local deduction theorem and equivalent algebraic semantics $\sf V$. Then $\vdash$ has the deductive interpolation property if and only if $\sf V$ has the amalgamation property.
\end{thm}

\section{Continuum-many extensions with the deductive interpolation property}

We will construct continuum-many subvarieties of each of $\RL$, $\PRL$, $\GA$, and $\GAM$ that have the amalgamation property. From Proposition~\ref{prop:isomorphism}, these subvarieties are in bijective correspondence with axiomatic extensions of $\crle$, $\cfle$, $\cmall$, and $\cll$, respectively. Consequently, we may conclude from Theorems~\ref{t:DIP-AP} and \ref{t:algebraizable} that there are continuum-many axiomatic extensions of each of $\crle$, $\cfle$, $\cmall$, and $\cll$ with the deductive interpolation property.
In each case, generating algebras for the aforementioned subvarieties are built from certain appropriately chosen abelian groups. Similar kinds of residuated lattices have been studied very recently in \cite{GZ23}.

For every prime number $p$ we let $\mathbf{Z}_p = \alg{Z_p, \cdot,{}^{-1}, 1}$ be the cyclic group of order $p$ and we let $\mathbf{Z} = \alg{\mathbb{Z},+,-,0}$ be the group of integers.
A subgroup $\mathbf{G}$ of a group $\mathbf{H}$ is called an \emph{essential subgroup} of $\mathbf{H}$ if for every non-trivial subgroup $\mathbf{G}'$ of $\mathbf{H}$ we have that $G\cap G'$ is non-trivial.

\begin{lemma}[\cite{Eckmann1953}]\label{l:embed-inj}
Every abelian group is an essential subgroup of an injective abelian group.
\end{lemma}

For each set $P$ of prime numbers we define the set of quasi-equations
\[
\Sigma_P = \set{x^p \approx 1 \Rightarrow x \approx 1 : p \in P}.
\]
Further, we denote by $\mathsf{Q}_P$ the quasivariety of abelian groups axiomatized by $\Sigma_P$. 

\begin{prop}\label{p:AP-groups}
For every set of prime numbers $P$, the quasivariety $\mathsf{Q}_P$ has the amalgamation property.
\end{prop}

\begin{proof}
By Lemma~\ref{l:EI-AP} it is enough to show that $\mathsf{Q}_P$ has enough injectives. Select $\mathbf{G}\in \mathsf{Q}_P$. By Lemma~\ref{l:embed-inj}, $\mathbf{G}$ is an essential subgroup of an injective abelian group $\mathbf{H}$. We want to show that $\mathbf{H} \in \mathsf{Q}_P$. Suppose for a contradiction that $\mathbf{H} \notin \mathsf{Q}_P$. Then there exists a prime $p\in P$ such that $\mathbf{H}\not\models x^p\approx 1 \Rightarrow x \approx 1$, i.e., there exists an $a \in H$ with $a^p = 1$ and  $a\neq 1$. Since the subgroup $\mathbf{S}$ generated by $a$ is a cyclic group of order $p$ and, since $\mathbf{G}$ is an essential subgroup of $\bf H$, there exists $b \in S\cap G$ with $b\neq 1$. Now, since $p$ is prime, $b^p = 1$, so $\mathbf{G} \notin \mathsf{Q}_P$, a contradiction. Hence we get $\mathbf{H} \in \mathsf{Q}_P$. 
\end{proof}

The following lemma shows that in certain situations it is possible to term-define an exponential on an A-algebra.

\begin{lemma}\label{l:girale}
Let $\bf A$ be an A-algebra such that $(a\meet 1)^2 = a \meet 1$ for all $a\in A$, and define $\bang\colon A\to A$ by $\bang a  = a \meet 1$. Then the algebra $\alg{\mathbf{A}, \bang}$ is a girale.
\end{lemma}

\begin{proof}
We have for every $a \in A$, $1\meet 1 = 1$, $a\meet 1 \leq a \meet 1$, and $(a \meet 1) \meet 1 = a \meet 1$. Hence $\alg{\mathbf{A}, \bang}$,  satisfies (G1), (G2), and (G4). Moreover, for (G3) note that, by assumption, for all $a,b \in A$, $a \meet 1 \leq 1$ and $b\meet 1 \leq 1$ are idempotent, so it follows  that $(a\meet 1)(b\meet 1) =(a\meet 1)\meet (b \meet 1) = (a\meet b) \meet 1$.
\end{proof}

We now introduce an algebraic construction that produces a commutative residuated lattice from an arbitrary abelian group ${\m G}$. The commutative residuated lattices arising from this construction may always be expanded to pointed commutative residuated lattices, A-algebras, and girales.

Let $\mathbf{G} = \alg{G,\cdot,{}^{-1},1}$ be an abelian group. We can consider it as a partially ordered group by defining  $a \leq_G b$ if $a = b$ for $a,b \in G$.\footnote{A \emph{partially ordered group} is a group equipped with a partial order such that the left and right translations are order preserving.}
By adding $\bot \notin G$ and and $\top \notin G$ as a new bottom and top element to $\alg{G,\leq_G}$, respectively, we obtain a bounded lattice $\alg{G\cup\set{\bot,\top},\meet, \join,\bot, \top}$. Figure~\ref{fig:3} depicts  the lattice-order for $\mathbf{G} = \mathbf{Z}_3$.

\begin{figure}
\centering
\begin{tikzpicture}[
place/.style={circle,draw=black,fill=black, minimum size = 4pt, inner sep = 0pt}]
  \node[place] (e) at (-1,0) {};
  \node[place] (a) at (0,0) {};
   \node[place] (a2) at (1,0) {};
  \node[place] (bot) at (0,-1) {};
   \node[place] (top) at (0,1) {};
  \node[left] () at (e) {$1$};
  \node[right] () at (a) {$a$};
   \node[right] () at (a2) {$a^2$};
  \node[below] () at (bot) {$\bot $};
  \node[above] () at (top) {$\top$};
  \draw (bot) -- (e) -- (top);
  \draw (bot) -- (a) -- (top);
  \draw (bot) -- (a2) -- (top);
\end{tikzpicture}
\caption{The lattice expansion for $\mathbf{Z}_3$}
\label{fig:3}
\end{figure}

Now if we extend the multiplication of $\bf G$ to $G\cup \set{\bot,\top}$ by stipulating  $a\cdot \top = \top \cdot a = \top$ for all $a \in G\cup \set{\top}$ and  $b \cdot \bot = \bot \cdot b = \bot$ for all $b\in G\cup \set{\bot,\top}$, we obtain a residuated, commutative, associative binary operation on $\alg{G\cup\set{\bot,\top},\meet, \join}$ with unit $1$. The residual $\to$ is uniquely determined by multiplication together with the lattice-order by the formula
\begin{equation}\label{eq:res}
a\to c = \max \{b : ab\leq c\ \}\tag{R}.
\end{equation}
Thus $\alg{G\cup \set{\bot,\top}, \meet, \join, \cdot, \to,1}$ is a commutative residuated lattice. Moreover, it is easy to see that $1$ is a negation constant, so $\alg{G\cup \set{\bot,\top}, \meet, \join, \cdot, \to,1,1,\bot,\top}$ is an A-algebra.
Also, for every $a \in G\cup \set{\bot,\top}$ we have $a \meet 1 \in \set{\bot,1}$, so clearly for all $a \in G$, we have $(a\meet 1)^2 = a\meet 1$. 
 Hence, by Lemma~\ref{l:girale}, we obtain a girale $\alg{G\cup \set{\bot,\top}, \meet, \join, \cdot, \to,1,1,\bot,\top, \bang}$ by defining $\bang a = a \meet 1$.
For $S \subseteq \set{0,\bot,\top, \bang}$, we denote by $\ext{\mathbf{G}}$ the $\mathcal{RL} \cup S$ reduct of this algebra. For example, for $S = \set{0}$, $\ext{\mathbf{G}}$ is the pointed residuated lattice $\alg{G\cup \set{\bot,\top}, \meet, \join, \cdot,\to,1,1}$. 

Recall that an arbitrary algebra $\m A$ is \emph{simple} if its only congruences are the equality relation $\Delta = \{\alg{x,y}\in A^2 : x = y\}$ and the equivalence relation that identifies all elements $\nabla = A^2$. For any commutative residuated lattice ${\m A}$, we may define another residuated lattice ${\m A}^-$ whose universe is the set of negative elements $A^- = \{a\in A : a\leq 1\}$. The operations $\meet,\join,\cdot,1$ on ${\m A}^-$ are inherited from ${\m A}$, and the residual $\to^-$ of ${\m A}^-$ is defined by the term $a\to^- b = (a\to b)\meet 1$, where $\to$ is the residual of ${\m A}$. The algebra ${\m A}^-$ is called the \emph{negative cone} of $\m A$. From \cite[Lemma 3.49]{GJKO07}, the lattice of congruences of any commutative residuated lattice ${\m A}$ is isomorphic to the lattice of congruences of its negative cone ${\m A}^-$.

The following lemma gives some elementary properties of $\ext{\mathbf{G}}$. 

\begin{lemma}\label{l:basicprop}
Let $\bf G$ be an abelian group and $S \subseteq \set{0,\bot,\top, \bang}$.
\begin{enumerate}[label = \textup{(\roman*)}]
\item $\ext{\mathbf{G}}$ is simple.
\item For every $a \in \extu{\mathbf{G}} \setminus\set{\top}$, $\top \to a = \bot$.
\item For all $a,b \in G$, $a \to b = a^{-1}b$.
\end{enumerate}
\end{lemma}

\begin{proof}
For (i), we observe that for $S=\emptyset$, the negative cone of $\ext{\mathbf{G}}$ is isomorphic to the two-element residuated lattice. Since this algebra is simple and the congruences of $\ext{\mathbf{G}}$ are in bijective correspondence with the congruences of its negative cone, the result follows for $S=\emptyset$. If $S\neq\emptyset$, we note that a congruence of $\ext{\mathbf{G}}$ is in particular a congruence of its commutative residuated lattice reduct, so the result follows from the previous comments.

For (ii), observe that if $a\neq\top$ then the largest element $b$ such that $b\cdot\top\leq a$ is $\bot$. Thus the claim follows from (\ref{eq:res}). Part (iii) likewise follows from direct computation with (\ref{eq:res}).
\end{proof}

For a set $P\neq \emptyset$ of prime numbers and  $S \subseteq \set{0,\bot,\top, \bang}$ we define the class
\[
\mathsf{K}^S_P = \iso(\set{\ext{\mathbf{G}} : {\m G}\in \mathsf{Q}_P} \cup \set{\mathbf{T}}),
\]
where $\mathbf{T}$ is a trivial algebra in the specified type. Further, we set $\mathsf{V}_P^S = \vr(\mathsf{K}_P^S)$. Moreover, we denote  by $\cllp$, $\cmallp$, $\cflep$, and $\crlep$ the extensions of  $\cll$, $\cmall$, $\cfle$, and $\crle$ corresponding to the varieties $\mathsf{V}_P^{ \set{0,\bot,\top, \bang}}$, $\mathsf{V}_P^{ \set{0,\bot,\top}}$,  $\mathsf{V}_P^{ \set{0}}$, and $\mathsf{V}_P^\emptyset$, respectively.  

Note that in the varieties $\mathsf{V}_P^S$, the operation $0$ is definable by $0 = 1$, $\top$ is definable by $\top = \bot \to 1$, $\bot$ is definable by $\bot = \top \to 1$,  and $\bang$ is definable by $\bang x = x \meet 1$. 
Hence, in what follows we can assume that either $S = \set{\bot,\top}$ or $S = \emptyset$.

\begin{prop}\label{p:HSPU}
For every set of prime numbers $P\neq \emptyset$ and  $S \subseteq \set{0,\bot,\top, \bang}$, the class $\mathsf{K}_P^S$ is a universal class, i.e.,  $\hm\sub\pu(\mathsf{K}_P^S) = \mathsf{K}_P^S$.
\end{prop}
\begin{proof}
Observe that $\mathsf{K}^{\set{\bot,\top}}_P$ is axiomatized relative to bounded commutative residuated lattices by the set of quasi-equations $\Sigma_P$ together the universal sentences 
\begin{align}
(\forall x)( (x\not\approx \bot)\andd (x\not\approx \top)   \implies (x(x\to 1) \approx 1)), \\
(\forall x)(\forall y) ((x\not\approx \bot) \andd (y \not \approx \bot) \andd (x\not \approx y) \implies (x\join y \approx \top)), \\
(\forall x)(\forall y) ((x\not\approx \top) \andd (y \not \approx \top) \andd (x\not \approx y) \implies (x\meet y \approx \bot)),\\
(\forall x)((x \not\approx \bot) \implies (x\cdot \top \approx \top)).
\end{align}
To see this let $\bf A$ be a non-trivial algebra satisfying these axioms. Then,  by (1), the set $A\setminus \set{\bot,\top}$ gives rise to an abelian group $\mathbf{G}_A$ with $a^{-1} = a\to 1$ for $a\in A \setminus \set{\bot,\top}$. Moreover, since $\bf A$ satisfies $\Sigma_P$, also $\mathbf{G}_A$ satisfies $\Sigma_P$, i.e., $\mathbf{G}_A \in \mathsf{Q}_P$.  Now, using (2), (3), and (4), it is straightforward to check that $\mathbf{A}$ is isomorphic to $\ext{\mathbf{G}_A}$.

For $\mathsf{K}^{\emptyset}_P$ we note that $\bot$ and $\top$ are definable by the formulas $\phi_\bot(x) = (\forall y)(x\meet y \approx x)$ and $\phi_\top(x) = (\forall y)(x \join y \approx x)$. So, the above sentences can be adapted to get an axiomatization of  $\mathsf{K}^{\emptyset}_P$, i.e., $\mathsf{K}^{\emptyset}_P$ is closed under $\pu$. Closure under $\hm\sub$ follows from Lemma~\ref{l:basicprop} (i) and the fact that $\mathsf{Q}_P$ is closed under subalgebras.
\end{proof}

\begin{cor}\label{c:ISP}
Let $P\neq \emptyset$ be a set of prime numbers and $S \subseteq \set{0,\bot,\top, \bang}$. Then:
\begin{enumerate}[label = \textup{(\roman*)}]
\item $\mathsf{V}_P^S = \iso\sub\pr(\mathsf{K}_P^S)$.
\item The class of finitely subdirectly irreducible members of $\mathsf{V}_P^S$ is exactly $\mathsf{K}_P^S$.
\end{enumerate}
\end{cor}

\begin{proof}
For (i), we observe that by Jónsson's Lemma together with Proposition~\ref{p:HSPU}, every subdirectly irreducible member of $\mathsf{V}_P^S$ is contained in $\mathsf{K}_P^S$. Hence, since every algebra in $\mathsf{V}_P^S$ is a isomorphic to a subdirect product of subdirectly irreducibles, we get $\mathsf{V}_P^S = \iso\sub\pr(\mathsf{K}_P^S)$. 
For part (ii) note that, since $\mathsf{V}_P^S = \iso\sub\pr(\mathsf{K}_P^S)$ by part (i), it follows from \cite[Lemma 1.5]{CD90} that every finitely subdirectly irreducible member of $\mathsf{V}_P^S$ is contained in $\iso\sub\pu(\mathsf{K}_P^S)$ and, by Proposition~\ref{p:HSPU}, we have $\iso\sub\pu(\mathsf{K}_P^S) = \mathsf{K}_P^S$. On the other hand every member of $\mathsf{K}_P^S$ is either trivial or simple, so in particular finitely subdirectly irreducible.
\end{proof}

\begin{lemma}\label{l:2}
Let $\mathbf{G}$ and $\mathbf{H}$ be abelian groups and $S \subseteq \set{0,\bot,\top, \bang}$.
\begin{enumerate}[label = \textup{(\roman*)}]
\item Every embedding $\alpha\colon \mathbf{G} \to \mathbf{H}$ uniquely extends to an embedding $\beta \colon \ext{\mathbf{G}} \to \ext{\mathbf{H}}$.
\item Every embedding $\beta\colon  \ext{\mathbf{G}} \to \ext{\mathbf{H}}$ restricts to an embedding $\alpha\colon \mathbf{G} \to \mathbf{H}$.
\end{enumerate}
\end{lemma}

\begin{proof}
For (i), let $\alpha\colon \mathbf{G} \to \mathbf{H}$ be an embedding, define  $\beta \colon \mathrm{R}^S(\mathbf{G}) \to \mathrm{R}^S(\mathbf{H})$ by $\beta(a) = \alpha(a)$ for $a\in G$ and $\beta(\top) = \top$, $\beta(\bot) = \bot$. Then clearly $\beta$ is a (bounded) lattice homomorphism and a monoid homomorphism. But also, by Lemma~\ref{l:basicprop}  (ii) and (iii),  $\top \to a = \bot$ for $a \in G\cup \set{\top}$ and $a\to b = a^{-1}b$ for $a,b \in G$. Moreover,  $\bot \to a = \top$ for $a \in G\cup\set{\bot,\top}$ and $a \to \bot = \bot$ for $a \in G\cup \set{\top}$. The same also holds for $\ext{\mathbf{H}}$. Hence, since $\alpha$ is a group homomorphism, $\beta$ preserves $\to$. That $\beta$ is the unique extension is clear, since any embedding from $\ext{\mathbf{G}}$ to $\ext{\mathbf{H}}$ needs to map $\bot$ to $\bot$ and $\top$ to $\top$. Part (ii) is immediate from Lemma~\ref{l:basicprop} (iii).
\end{proof}

\begin{thm}\label{t:ap}
For every set of prime numbers $P \neq \emptyset$ and $S \subseteq \set{0,\bot,\top, \bang}$, the variety $\mathsf{V}_P^S$ has the amalgamation property. 
\end{thm}

\begin{proof}
By Corollary~\ref{c:ISP}(ii), the class of finitely subdirectly members of $\mathsf{V}_P^S$ is exactly $\mathsf{K}_P^S$ and, by Proposition~\ref{p:HSPU}, it is closed under subalgebras. Since $\mathsf{V}_P^S$ is term-equivalent to a variety of (possibly bounded or pointed) commutative residuated lattices, $\mathsf{V}_P^S$ has the congruence extension property, by \cite[Lemma 3.57]{GJKO07}. Therefore, from Theorem~\ref{t:lift amalg} it suffices to show that $\mathsf{K}_P^S$ has the amalgamation property. 

It follows from Lemma~\ref{l:2} together with Proposition~\ref{p:AP-groups} that any span in $\mathsf{K}_P^S$ that does not contain a trivial algebra has an amalgam in $\mathsf{K}_P^S$, since we can lift the amalgam of the span of the group subreducts. If $S = \set{\bot,\top}$, a trivial algebra in $\mathsf{K}_P^S$ does only embed into trivial algebras, so we are done. If $S = \emptyset$ we note that in every span in $\mathsf{K}_P^S$ that contains trivial algebras, we can replace the trivial algebras with the algebra $\ext{\mathbf{0}}$ where $\bf 0$ is a trivial abelian group and extend the embeddings accordingly. Hence the claim also follows for $S = \emptyset$.
\end{proof}

By Theorem~\ref{t:DIP-AP} we get the following result about the corresponding axiomatic extensions.
\begin{cor}
For every set of prime numbers $P\neq \emptyset$, each of $\cllp$, $\cmallp$, $\cflep$, and $\crlep$ has the deductive interpolation property.
\end{cor}
The previous proposition shows that the extension we have identified have the \prp{DIP}. It remains to show that there are continuum-many of them.
\begin{prop}\label{p:primes distinct}
For two non-empty sets  $P$ and $P'$ of prime numbers with $P\neq P'$ and  $S \subseteq \set{0,\bot,\top, \bang}$  we have $\mathsf{V}_P^S \neq \mathsf{V}_{P'}^S$. 
\end{prop}
\begin{proof}
If $P \neq P'$, then without loss of generality there is a $p\in P$ with $p\notin P'$. Since we have  $\mathbf{Z}_p\not \models x^p \approx 1 \Rightarrow x \approx 1$  and $\mathbf{Z}_p \models x^q \approx 1 \Rightarrow x \approx 1$ for every prime  $q \neq p$,  we get $\mathbf{Z}_p \notin \mathsf{Q}_P$ and $\mathbf{Z}_p \in \mathsf{Q}_{P'}$. But then also $\ext{\mathbf{Z}_p} \notin \mathsf{K}_P^S$ and $\ext{\mathbf{Z}_p} \in \mathsf{K}_{P'}^S$. Hence, by Proposition~\ref{p:HSPU} together with Jónsson's Lemma, we get  $\ext{\mathbf{Z}_p} \notin \mathsf{V}_P^S$ and $\ext{\mathbf{Z}_p} \in \mathsf{V}_{P'}^S$, i.e., $\mathsf{V}_P^S \neq \mathsf{V}_{P'}^S$.
\end{proof}

We finally arrive at the main result of this paper:

\begin{thm}\label{t:main}
\begin{enumerate}[label = \textup{(\roman*)}]
\item[]
\item Each of $\RL$, $\PRL$, $\GA$, and $\GAM$ has continuum-many subvarieties with the amalgamation property.
\item Each of $\cll$, $\cmall$, $\cfle$, and $\crle$ has continuum-many axiomatic extensions with the deductive interpolation property.
\end{enumerate}
\end{thm}

\begin{proof}
(i) is immediate from Theorem~\ref{t:ap} and Proposition~\ref{p:primes distinct} because there are continuum-many distinct sets of prime numbers. (ii) then follows from (i) by Theorem~\ref{t:algebraizable} and Proposition~\ref{prop:isomorphism}.
\end{proof}

The continuum-many logics we have identified have deductive interpolation, but they all provably lack the \prp{CIP}. To see this, we scrutinize the strong amalgamation property in the corresponding varieties.

\begin{prop}\label{p:SE}
For every non-empty set of prime numbers $P$ and  $S \subseteq \set{0,\bot,\top, \bang}$, the variety $\mathsf{V}_P^S$ does not have surjective epimorphisms and hence lacks the strong amalgamation property.
\end{prop}

\begin{proof}
Let $P$ be a non-empty set of prime numbers and $p\in P$. Then we have $\mathbf{Z} \in \mathsf{Q}_P$, i.e., $\ext{\mathbf{Z}} \in \mathsf{K}_P^S$. Consider the embedding $\phi\colon \ext{\mathbf{Z}} \to \ext{\mathbf{Z}}$ defined by 
\[
\phi(a) = 
\begin{cases}
a &\text{if } a \in \set{\bot,\top} \\
pa& \text{if } a\in \mathbb{Z}.
\end{cases}
\]
Clearly $\phi$ is not surjective. Let $\mathbf{A} \in \mathsf{V}_P^S$ with identity $e$ and let  $\psi_1,\psi_2 \colon \ext{\mathbf{Z}} \to \mathbf{A}$ be homomorphisms satisfying $\psi_1 \circ \phi = \psi_2 \circ \phi$. Then, since $\ext{\mathbf{Z}}$ is simple,  either both $\psi_1$ and $\psi_2$ are constant or $\psi_1$ and $\psi_2$ are embeddings. In the first case we clearly have $\psi_1 = \psi_2$, so we may assume that the maps are embeddings. Since $1$ generates $\ext{\mathbf{Z}}$ it suffices to show that $\psi_1(1) = \psi_2(1)$.  Now, by assumption, $(\psi_1(1)\cdot \psi_2(-1))^p =  \psi_1(p) \cdot \psi_2(-p) = e$. But, since, $\mathsf{V}_P^S \models x^p \approx 1 \Rightarrow x \approx 1$, we get $\psi_1(1) \cdot \psi_2(-1) = e$, i.e., $\psi_1(1) = \psi_2(1)$. Hence, $\psi_1 = \psi_2$.
\end{proof}

\begin{remark}
The proof of Proposition~\ref{p:SE} illustrates that $p$th roots are implicitly definable but not explicitly definable in the varieties $\mathsf{V}_P^S$. This reflects the connection between epimorphism-surjectivity and the Beth definability property; see, e.g., \cite{H2000}. We do not, however, further discuss Beth definability here.
\end{remark}

\begin{cor}
For every non-empty set of prime numbers $P$ the extensions $\cllp$, $\cmallp$, $\cflep$, and $\crlep$ do not have the Craig interpolation property.
\end{cor}

\begin{proof}
Let $P$ be a non-empty set of prime numbers and  $\vdash^P$ be one of the extensions $\cllp$, $\cmallp$, $\cflep$, or $\crlep$ with $\mathsf{V}_P^S$ its equivalent algebraic semantics. If $\vdash^P$ had the \prp{CIP}, then $\mathsf{V}_P^S$ would have the strong amalgamation property; see, e.g., \cite{GJKO07} and note that, since we can assume that $S= \set{\bot,\top}$ or $S= \emptyset$, the results therein for axiomatic extensions of $\fle$ apply. But then Lemma~\ref{l:ES-SAP} would imply that $\mathsf{V}_P^S$  has surjective epimorphisms, contradicting Proposition~\ref{p:SE}.
\end{proof}

\begin{remark}
The results of this section entail that various other substructural logics have continuum-many extensions with the deductive interpolation property, but without the Craig interpolation property. This applies, for example, to the deductive system for involutive full Lambek calculus with exchange $\m{InFL}_e$ and the deductive system for full Lambek calculus with exchange and right-weakening $\m{FL}_{eo}$ (cf. \cite{GJKO07}). Moreover, since linear logic and multiplicative-additive linear logic are axiomatic extensions of intuitionistic linear logic and intuitionistic multiplicative-additive linear logic, respectively, the results also follow for these logics (see, e.g., \cite{Okada1999} for a definition of intuitionistic linear logic and its fragments). 
\end{remark}

Although the logics we have identified do not have the \prp{CIP}, they have a weak form of interpolation with respect to $\to$. We say that a logic $\vdash$ has the \emph{guarded interpolation property} if whenever $\vdash\bang\varphi\to\bang \psi$, there exists a formula $\delta$ whose variables are among those contained in both $\varphi$ and $\psi$ such that $\vdash\bang\varphi\to\bang\delta$ and $\vdash \bang\delta\to\bang\psi$. 

\begin{lemma}\label{l:DDP}
For every set of formulas $\Gamma\cup\set{\phi,\psi}$,
\[
\Gamma, \phi \cll \psi \iff \Gamma \cll \bang\phi \to \bang \psi.
\]
\end{lemma}

\begin{proof}
The equivalence follows from Avron's deduction theorem together with the fact that, by algebraization and since $\bang$ is idempotent, order preserving, and contracting in every girale, $\cll \bang \psi \to \psi$ and if $\Gamma \cll \bang\phi \to \psi$, then $\Gamma \cll \bang \phi \to \bang \psi$.
\end{proof}

From Remark~\ref{r:DIP-one-formula} together with Lemma~\ref{l:DDP} we get:

\begin{prop}\label{p:guarded-craig}
An axiomatic extension $\vdash$  of $\cll$ has the deductive interpolation property if and only if it has the guarded interpolation property.
\end{prop}

\begin{remark}
Proposition~\ref{p:guarded-craig} is just a special case of a general observation of \cite[Section 4.4]{CP99} about conjunctive deductive systems with a deduction-detachment theorem ($\prp{DDT}$). In view of Proposition~\ref{p:guarded-craig} the witness for the $\prp{DDT}$ that we consider is the formula $ \bang x \to \bang y$. However, considering the formula $\bang x \to y$  of Avron's deduction theorem yields yet another equivalent interpolation property.
\end{remark}

\begin{cor}\label{t:modal-craig}
For every set of prime numbers $P \neq \emptyset$ the axiomatic extension $\cllp$ has the guarded interpolation property.
\end{cor}

\section{Continuum-many failures of the deductive interpolation property}
We conclude by showing that the deductive interpolation property fails for continuum-many axiomatic extensions of $\cll$. Our approach is to show that the amalgamation property fails for continuum-many varieties of girales, where we obtain these varieties by adding a finite algebra to the class of generating algebras of each of the continuum-many varieties constructed in the previous section.

Sugihara monoids \cite{FG2019} have been studied extensively as algebraic models of certain relevance logics, and we will exploit them here in order to construct the varieties of girales just mentioned. For $n\in \mathbb{N}$, define
$$C_{2n+1} = \set{-n,-n+1,\cdots, 0, \cdots, n-1, n}.$$
The $(2n+1)$-element odd Sugihara chain is the totally ordered commutative residuated lattice $\mathbf{C}_{2n+1}= \alg{C_{2n+1}, \meet,\join,\cdot,\to, 0}$, where $x \meet y = \min(x,y)$, $x \join y = \max(x,y)$,  
\[
x\cdot y = \begin{cases}
x \meet y & \text{if } \lvert x \rvert  = \lvert y \rvert \\
y 	           & \text{if } \lvert x \rvert  < \lvert y \rvert \\
x	           & \text{if } \lvert y \rvert  < \lvert x \rvert \\
\end{cases}
\quad
\text{and}
\quad
x \to y = \begin{cases}
(-x) \join y & \text{if } x \leq y \\
(-x) \meet y & \text{otherwise},
\end{cases}
\]
where $\lvert x \rvert$ is the absolute value of $x$. Note that $0$ is a negation constant satisfying $x\to 0 = -x$ for each $x\in C_{2n+1}$. Moreover, for every $x\in C_{2n+1}$, we have $x\cdot x = x\meet x = x$. Thus, by Lemma~\ref{l:girale}, the Sugihara chain $\mathbf{C}_{2n+1}$ gives rise to a girale $\alg{C_{2n+1}, \meet,\join,\cdot, \to, 0, 0, -n, n,\bang}$ with $\bang x = x \meet 0$. As in the previous section, for every $S\subseteq \set{0,\bot,\top, \bang}$ we denote by $\mathbf{C}_{2n+1}^S$ its respective reduct. For $S \subseteq \set{0,\bot,\top, \bang}$ we define $\mathsf{K}_7^S = \iso(\set{\mathbf{C}_1^S, \mathbf{C}_3^S,\mathbf{C}_5^S,\mathbf{C}_7^S})$ and for a non-empty set of prime numbers $P$ we set $\mathsf{W}^S_P = \vr( \mathsf{K}_P^S \cup \mathsf{K}_{7}^S)$. We denote  by $\clls$, $\cmalls$, $\cfles$, and $\crles$ the axiomatic extensions of  $\cll$, $\cmall$, $\cfle$, and $\crle$ corresponding to the varieties $\mathsf{W}_P^{ \set{0,\bot,\top, \bang}}$, $\mathsf{W}_P^{ \set{0,\bot,\top}}$,  $\mathsf{W}_P^{ \set{0}}$, and $\mathsf{W}_P^\emptyset$, respectively.  
Note that, for the same reason as in the last section, in what follows we can again assume that either $S= \set{\bot,\top}$ or $S =\emptyset$.

\begin{lemma}\label{l:sugihara-FSI}
Let $P$ be a non-empty set of prime numbers and $S\subseteq \set{0,\bot,\top, \bang}$. 
\begin{enumerate}[label = \textup{(\roman*)}]
\item $\hm\sub\pu(\mathsf{K}_P^S \cup \mathsf{K}_7^S) = \mathsf{K}_P^S \cup \mathsf{K}_7^S$.
\item The class of finitely subdirectly irreducible members of $\mathsf{W}_P^S$ is  $\mathsf{K}_P^S \cup \mathsf{K}_7^S$.
\end{enumerate}
\end{lemma}

\begin{proof}
For (i), we observe that, by Proposition~\ref{p:HSPU} and since up to isomorphism $\mathsf{K}_7^S$ consists of only finitely many finite algebras,  $\hm\sub\pu(\mathsf{K}_P^S \cup \mathsf{K}_7^S)  = \mathsf{K}_P^S \cup \hm\sub(\mathsf{K}_7^S)$. A straightforward calculation shows that $\mathsf{K}_7^S = \hm\sub(\mathbf{C}_7^S)$, yielding  $\hm\sub(\mathsf{K}_7^S) = \mathsf{K}_7^S$. 

For (ii) note that, by Jónsson's Lemma together with part (i),  every subdirectly irreducible member of $\mathsf{W}_P^S$ is contained in $\mathsf{K}_P^S \cup \mathsf{K}_7^S$, so  $\mathsf{W}_P^S = \iso\sub\pr(\mathsf{K}_P^S \cup \mathsf{K}_7^S)$. 
Hence, by \cite[Lemma 1.5]{CD90}, each finitely subdirectly irreducible member of $\mathsf{W}_P^S$ is contained in $\iso\sub\pu(\mathsf{K}_P^S \cup \mathsf{K}_7^S) = \mathsf{K}_P^S \cup \mathsf{K}_7^S$.  
On the other hand, by Corollary~\ref{c:ISP} and since any non-trivial finite, totally ordered commutative residuated lattice is subdirectly irreducible, each algebra in $\mathsf{K}_P^S \cup \mathsf{K}_7^S$ is finitely subdirectly irreducible.
\end{proof}

The next result shows that adding the algebra $\mathbf{C}_7^S$ to the generators of $\mathsf{V}_P^S$ results in a variety that fails to have the amalgamation property.

\begin{thm}
For any non-empty set of prime numbers $P$ and $S \subseteq\set{0,\bot,\top, \bang}$ the variety $\mathsf{W}_P^S$ does not have the amalgamation property. 
\end{thm}

\begin{proof}
By Lemma~\ref{l:sugihara-FSI}(ii), the class of finitely subdirectly members of $\mathsf{W}_P^S$ is exactly $\mathsf{K}_P^S \cup \mathsf{K}_7^S$ and, by Lemma~\ref{l:sugihara-FSI}(i), it is closed under subalgebras. Moreover, as again $\mathsf{W}_P^S$ is term-equivalent to a variety of (possibly bounded or pointed) commutative residuated lattices $\mathsf{W}_P^S$ has the congruence extension property, by \cite[Lemma 3.57]{GJKO07}. Hence, by Theorem~\ref{t:lift amalg}, it is enough to show that $\mathsf{K}_P^S \cup \mathsf{K}_7^S$ does not have the one-sided amalgamation property. To see this, we consider the span  $\alg{\mathbf{C}_5^S,\mathbf{C}_7^S,\mathbf{C}_7^S, \phi_1 ,\phi_2}$, where $\phi_1(-2) =\phi_2(-2) = -3$, $\phi_1(0) = \phi_2(0) = 0$, $\phi_1(2) = \phi_2(2) = 3$, $\phi_1(-1) = -2$, $\phi_1(1) = 2$, $\phi_2(-1) = -1$, $\phi_2(1) = 1$. Suppose for a contradiction that the span has a one-sided amalgam $\alg{\psi_1,\psi_2, \mathbf{D}}$ in $\mathsf{K}_P^S \cup \mathsf{K}_7^S$. 
Note that up to isomorphism $\mathbf{C}_7^S$ is the only member of  $\mathsf{K}_P^S \cup \mathsf{K}_7^S$ of height $7$ or higher. But, since $\psi_1$ is an embedding of $\mathbf{C}_7^S$ into $\bf D$, it follows that $\bf D$ needs to be of height $7$ or higher, so without loss of generality we can assume that $\mathbf{D} = \mathbf{C}_7^S$ and $\psi_1$ is the identity map.  Hence we have $\psi_2(-1) = \psi_2(\phi_2(-1)) =  \phi_1(-1) = -2$ and $\psi_2(-3) =\psi_2(\phi_2(-2)) =  \phi_1(-2) = -3$. Thus, since $\psi_2(-3)\leq \psi_2(-2) \leq \psi_2(-1)$, either $\psi_2(-2) = \psi_2(-3)$ or $\psi_2(-2) = \psi(-1)$. The first case yields
 \[
 -3 = -3 \cdot 3 = \psi_2(-2) \cdot \psi_2(3) = \psi_2(-2 \cdot 3) = \psi_2(3) = 3,
 \]
 a contradiction. For the second case, note that if $\psi_2(-2) = \psi_2(-1) = -2$, then also  $\psi_2(2) = \psi_2(1) = 2$, since $\phi$ preserves the involution. Thus
 \[
 -2 = -2 \cdot 2 = \psi_2(-1) \cdot \psi_2(2) = \psi_2(-1 \cdot 2) = \psi_2(2) = 2,
 \]  
 which is again a contradiction.
\end{proof}
Theorem~\ref{t:DIP-AP} yields the following corollary about the corresponding axiomatic extensions.
\begin{cor}
For any non-empty set of prime numbers $P$ the extensions $\clls$, $\cmalls$, $\cfles$, and $\crles$ do not have the deductive interpolation property.
\end{cor}

\begin{prop}\label{p:primes-distinct2}
For two non-empty sets  $P$ and $P'$ of prime numbers with $P\neq P'$ and  $S \subseteq \set{0,\bot,\top, \bang}$  we have $\mathsf{W}_P^S \neq \mathsf{W}_{P'}^S$. 
\end{prop}
\begin{proof}
If $P \neq P'$, then, by Proposition~\ref{p:primes distinct}, $\mathsf{V}_P^S \neq \mathsf{V}_{P'}^S$. Without loss of generality we may assume that  $\mathsf{V}_P^S \nsubseteq \mathsf{V}_{P'}^S$. Hence, by  Corollary~\ref{c:ISP}(ii),  there exists an algebra $\mathbf{A} \in \mathsf{K}_P^S$ with $\mathbf{A} \notin \mathsf{K}_{P'}^S$.  Moreover, the only member of $\mathsf{K}_7^S$  that is of height $3$ is $\mathbf{C}_3^S$ which is isomorphic to $\ext{\mathbf{0}}$, where $\bf 0$ is a trivial group. Hence, since all non-trivial members of ${\mathsf{K}_{P'}^S}$ have height 3, we get  $\mathbf{A} \in \mathsf{K}_{P}^S \cup \mathsf{K}_7^S$, $\mathbf{A} \notin \mathsf{K}_{P'}^S \cup \mathsf{K}_7^S$ and  Lemma~\ref{l:sugihara-FSI}(ii) yields that $\mathsf{W}_{P}^S \neq \mathsf{W}_{P'}^S$. 
\end{proof}

\begin{thm}\label{t:noDIP}
\begin{enumerate}[label = \textup{(\roman*)}]
\item[]
\item Each of $\RL$, $\PRL$, $\GA$, and $\GAM$ has continuum-many subvarieties without the amalgamation property.
\item Each of $\cll$, $\cmall$, $\cfle$, and $\crle$ has continuum-many axiomatic extensions without the deductive interpolation property.
\end{enumerate}
\end{thm}
\begin{proof}
(i) is immediate from Proposition~\ref{p:primes-distinct2} and the fact that there are continuum-many sets of prime numbers and (ii) follows from (i) by Theorem~\ref{t:algebraizable} and Proposition~\ref{prop:isomorphism}.
\end{proof}

\begin{remark}
Note that for $\cfle$ and $\crle$ the result already follows from \cite{Mak77}, in which it is shown that of the continuum-many superintuitionistic logics only finitely many have the deductive interpolation property. However, as superintuitionistic logics do not have an involutive negation, the result is new for $\cll$ and $\cmall$. Again, a similar result is also entailed for various other substructural logics, e.g., for the deductive systems of $\m{InFL}_e$ and $\m{FL}_{eo}$.
\end{remark}

\bibliographystyle{plain}
\bibliography{bibliography}

\end{document}